\documentclass{CSML}
\pdfoutput=1

\usepackage{lastpage}
\newcommand{\dOi}{14(1:13)2018}
\lmcsheading{}{1--\pageref{LastPage}}{}{}%
{May~12, 2013}{Feb.~06, 2018}{}

\usepackage{amssymb}
\usepackage{enumitem}
\usepackage{hyperref}
\hypersetup{hidelinks}

\begin{document}

\def\cal#1{{\mathcal #1}}
\def\bsigma#1{\mathbf{\Sigma}^0_{#1}}
\def\bpi#1{\mathbf{\Pi}^0_{#1}}
\def\bdelta#1{\mathbf{\Delta}^0_{#1}}
\def\bdiff#1{\mathcal{D}_{#1}}
\def\baire{\omega^{\omega}}
\def\mybf#1{{\mathbf #1}}


\def\cal#1{{\mathcal #1}}
\def\bb#1{{\mathbb #1}}
\def\bsigma#1{\mathbf{\Sigma}^0_{#1}}
\def\bpi#1{\mathbf{\Pi}^0_{#1}}
\def\bdelta#1{\mathbf{\Delta}^0_{#1}}
\def\analytic{\mathbf{\Sigma}^1_1}
\def\coanalytic{\mathbf{\Pi}^1_1}
\def\bianalytic{\mathbf{\Delta}^1_1}
\def\baire{\omega^{\omega}}
\def\cntbased{$\omega${\bf Top}$_0$}
\def\bdiff#1{\mathbf{\Sigma}^{-1}_{#1}}

\def\SpecZ{{\Lambda}}

\title{A generalization of a theorem of Hurewicz for quasi-Polish spaces}
\author{Matthew de Brecht}
\address{Graduate School of Human and Environmental Studies, Kyoto University, Kyoto, Japan}
\email{matthew@i.h.kyoto-u.ac.jp}


\begin{abstract}
We identify four countable topological spaces $S_2$, $S_1$, $S_D$, and $S_0$ which serve as canonical examples of topological spaces which fail to be quasi-Polish. These four spaces respectively correspond to the $T_2$, $T_1$, $T_D$, and $T_0$-separation axioms.  $S_2$ is the space of rationals, $S_1$ is the natural numbers with the cofinite topology, $S_D$ is an infinite chain without a top element, and $S_0$ is the set of finite sequences of natural numbers with the lower topology induced by the prefix ordering. Our main result is a generalization of Hurewicz's theorem showing that a co-analytic subset of a quasi-Polish space is either quasi-Polish or else contains a countable $\bpi 2$-subset homeomorphic to one of these four spaces.
\end{abstract}

\dedicatory{This paper is dedicated to Dieter Spreen in celebration of his 65th birthday.}
\maketitle

\section{Introduction}

This paper is a continuation of recent work on developing the descriptive set theory of general topological spaces initiated by V. Selivanov (see \cite{selivanov}). It was recently shown in \cite{debrecht2013} that a very general class of countably based topological spaces, called \emph{quasi-Polish spaces}, allow a natural extension of the descriptive set theory of Polish spaces (see \cite{kechris})  to the non-metrizable case. The class of quasi-Polish spaces contains not only the class of Polish spaces, but also many non-Hausdorff spaces that occur in fields such as theoretical computer science (e.g., $\omega$-continuous domains with the Scott-topology) and algebraic geometry (e.g., the spectrum of a countable commutative ring with the Zariski topology).

Given that so many important classes of countably based spaces happen to be quasi-Polish, it would be nice to have examples and characterizations of the spaces that are \emph{not} quasi-Polish. In the traditional descriptive set theory for Polish spaces, the space of rationals is the typical example of a metrizable space that is not Polish. In fact, an important theorem by W. Hurewicz \cite{hurewicz} states that a co-analytic subset of a Polish space is either Polish or else it contains a closed subset homeomorphic to the rationals. This shows that the rationals are the ``canonical'' example of a non-Polish metrizable space, since any ``definable'' example of a non-Polish metrizable space will contain the rationals as a closed subset. The reader can find a proof of Hurewicz's original theorem and some generalizations in \cite{kechris} (Theorem 21.18), \cite{solecki1994} (Corollary 8), and \cite{louveau_saintraymond}. 

The purpose of this paper is to identify the canonical examples of non-quasi-Polish spaces, and to extend Hurewicz's theorem to this more general setting. Naturally, the rationals alone will be insufficient for our purposes, but it turns out that a characterization is possible using only four different spaces. Here are the four canonical examples of countably based $T_0$-spaces that are not quasi-Polish:
\begin{itemize}
\item
The space $S_2$, defined as the set of rational numbers with the subspace topology inherited from the space of reals. This space satisfies the $T_2$-separation axiom, meaning that distinct points can be separated by disjoint open sets.
\item
The space $S_1$, defined as the set of natural numbers with the cofinite topology. A non-empty subset of this space is open if and only if it is cofinite. This is a $T_1$-space (i.e., every singleton subset is closed), but it is not a $T_2$-space.
\item
The space $S_D$, defined as the set of natural numbers with the Scott-topology under the usual ordering. A non-empty subset of this space is open if and only if it is of the form $\uparrow\!n=\{m\in\omega\,|\,n\leq m\}$ for some natural number $n\in\omega$. This is not a $T_1$-space, but it satisfies a weak separation axiom due to Aull and Thron \cite{aull_thron} called the $T_D$-separation axiom. A space is a $T_D$-space if and only if every singleton subset is \emph{locally closed} (i.e., equal to the intersection of an open set with a closed set).  
\item
The space $S_0$, defined as the set of finite sequences of natural numbers, ordered by the prefix relation, and given the lower topology. A subbasis for the \emph{closed} subsets of this space is given by sets of the form $\uparrow\!p=\{q\in\omega^{<\omega}\,|\, p\preceq q\}$, where $p\in\omega^{<\omega}$ and $\preceq$ is the prefix relation. This space fails to satisfy the $T_D$-separation axiom (every non-empty locally closed subset of this space is infinite), but it does satisfy the $T_0$-separation axiom.
\end{itemize}
Our main result (see Theorem \ref{thrm:main_theorem} below) is that a co-analytic subset of a quasi-Polish space is either quasi-Polish or else it contains a $\bpi 2$-subset homeomorphic to either $S_0$, $S_D$, $S_1$, or $S_2$. The majority of this paper concerns proving the statement for countable spaces, and then we will apply Hurewicz's original theorem to obtain the general result. 

As mentioned earlier, the space of rationals, $S_2$, is well known as the typical example of a non-Polish metric space. The spaces $S_1$ and $S_D$ are only slightly less well known, and also appear in several places as counter examples. For example, $S_1$ is the typical example of a non-sober $T_1$-space, and $S_D$ is the typical example of a partially ordered set that is not directed-complete. Note that $S_D$ also fails to be sober, and we will see that $S_D$ and $S_1$ are the canonical examples of non-sober countably based $T_0$-spaces (see Theorem~\ref{thrm:sober_characterization} below). 

Together, the three spaces $S_D$, $S_1$, and $S_2$ are the canonical examples of countable perfect $T_D$-spaces, and provide the main examples for why a space can fail to be a \emph{Baire space} (i.e., a space in which every countable intersection of dense open sets is dense). In fact, one of our results (see Theorem \ref{thrm:completelyBaire_characterization} below) shows that a countably based $T_0$-space is \emph{completely Baire} (i.e., every closed set is a Baire space) if and only if it does not contain a $\bpi 2$-subset homeomorphic to one of these three countable perfect $T_D$-spaces. The completely Baire property will be very useful for the purposes of this paper, but it is interesting to note that it is relevant even in fields such as locale theory, for example in determining the spatiality of locale presentations \cite{heckmann} and localic products \cite{tillplewe}.

The author is unaware of any previous references to the space $S_0$, but this space is also useful as a counter example in many cases. For example, $S_0$ is an example of a countable completely Baire space which is not quasi-Polish. This is a rare example where the descriptive set theories for Polish spaces and quasi-Polish spaces diverge, since any countable completely Baire metrizable space is Polish.

Furthermore, as a result of our Theorem \ref{thrm:CharacterizationOfCountableDelta3} below, $S_0$ cannot be embedded into a quasi-Polish space as a $\bpi 3$-subset. It follows from the results of \cite{debrecht2013} that $S_0$ is a countable space which does not admit a bicomplete quasi-metric (the only other typical example of such a countable space that the author is aware of is the set of rationals with the Scott-topology under the usual ordering).


It is also interesting to note that $S_0$ and $S_2$ are the canonical examples of countable sober spaces which are not quasi-Polish. From R. Heckmann's characterization of countably presentable locales in \cite{heckmann}, we can conclude that the topologies of $S_0$ and $S_2$ do not have countable presentations.

This paper is composed of seven sections. We provide some preliminary definitions in the next section. Each of the later sections is concerned with extending our characterizations of non-quasi-Polish spaces to more general classes of spaces, leading up to our final characterization of all co-analytic spaces in the last section.

\section{Preliminaries}

Our notation will follow that of \cite{debrecht2013}. The reader is expected to be familiar with general topology, descriptive set theory \cite{kechris}, and domain theory \cite{etal_scott}. The following modification of the Borel hierarchy, due to V. Selivanov, is required in order to provide a meaningful classification of the Borel subsets of non-metrizable spaces.

\begin{defi}
Let $X$ be a topological space. For each ordinal $\alpha$ ($1\leq \alpha < \omega_1$) we define $\bsigma \alpha(X)$ inductively as follows.
\begin{enumerate}
\item
$\bsigma 1(X)$ is the set of all open subsets of $X$.
\item
For $\alpha>1$, $\bsigma \alpha(X)$ is the set of all subsets $A$ of $X$ which can be expressed in the form
\[A=\bigcup_{i\in\omega}B_i\setminus B'_i,\]
where for each $i$, $B_i$ and $B'_i$ are in $\bsigma {\beta_i}(X)$ for some $\beta_i<\alpha$.
\end{enumerate}
We define $\bpi \alpha(X)= \{X\setminus A\,|\, A\in \bsigma \alpha(X)\}$ and $\bdelta \alpha(X)=\bsigma \alpha(X)\cap \bpi \alpha(X)$.
\qed
\end{defi}

The above definition is equivalent to the classical definition of the Borel hierarchy for metrizable spaces, but it differs for more general spaces. 

A topological space is \emph{quasi-Polish} if and only if it is countably based and admits a Smyth-complete quasi-metric. Polish spaces and $\omega$-continuous domains are examples of quasi-Polish spaces. A space is quasi-Polish if and only if it is homeomorphic to a $\bpi 2$-subset of $\cal P(\omega)$, the power set of $\omega$ with the Scott-topology. The reader should consult \cite{debrecht2013} for additional results on quasi-Polish spaces.

Since we will only be concerned with countably based spaces, the following notation will be convenient.

\begin{defi}\label{def:basis}
Given a countably based space $X$ with a basis $\{B_i\}_{i\in\omega}$ of open sets, we define $B(x,n)=\bigcap\{B_i\,|\, x\in B_i \mbox{ and } i\leq n\}$ for each $x\in X$ and $n\in\omega$.
\qed
\end{defi}

In the above definition we are using the convention that the empty intersection equals $X$, so $B(x,n)=X$ if there is no $i\leq n$ with $x\in B_i$. Note that for any open $U$ containing $x$, there is $n\in\omega$ with $x\in B(x,n)\subseteq U$.

We will also often refer to the specialization order on a topological space.

\begin{defi}
Given a topological space $X$, the \emph{specialization order} $\leq$ on $X$ is defined as $x\leq y$ if and only if $x$ is in the closure of $y$.
\qed
\end{defi}

The specialization order on $X$ is a partial order if and only if $X$ is a $T_0$-space. The specialization order is a trivial partial order if and only if $X$ is a $T_1$-space.

\section{\texorpdfstring{Countable perfect $T_D$-spaces}{Countable perfect TD-spaces}}\label{sec:countableTDspaces}

A subset of a space is \emph{locally-closed} if it is equal to the intersection of an open set with a closed set. A \emph{$T_D$-space} is a space in which every singleton subset is locally closed. A space is \emph{perfect} if and only if every non-empty open subset is infinite. Note that if $X$ is a $T_0$-space, then $X$ is perfect if and only if there is no $x\in X$ such that the singleton subset $\{x\}$ is open. 

Our interest in countable perfect $T_D$-spaces is due to the fact that a countably based countable $T_D$-space is quasi-Polish if and only if it does not have a non-empty perfect subspace (see the section on scattered spaces in \cite{debrecht2013}). 

In order to help familiarize the reader with $T_D$-spaces, we first prove a characterization of countably based countable $T_D$-spaces in terms of the Borel complexity of the diagonal (recall that the \emph{diagonal} of a space $X$ is defined as $\Delta_X=\{\langle x,y\rangle\in X\times X\,|\, x=y\}$).

\begin{thm}
The following are equivalent for a countably based space $X$ with countably many points:
\begin{enumerate}
\item
$X$ is a $T_D$-space,
\item
Every singleton subset $\{x\}$ of $X$ is in $\bdelta 2(X)$,
\item
Every subset of $X$ is in $\bdelta 2(X)$,
\item
The diagonal of $X$ is in $\bdelta 2(X\times X)$.
\end{enumerate}
\end{thm}
\begin{proof}
$(1 \Rightarrow 2)$. Easily follows from the definition of the $T_D$-axiom because locally closed sets are $\bdelta 2$.

$(2 \Rightarrow 3)$. If every singleton subset of $X$ is $\bdelta 2$, then the countability of $X$ implies that every subset of $X$ is the countable union of $\bdelta 2$-sets. Thus for any $S\subseteq X$ both $S$ and the complement of $S$ are $\bsigma 2$, hence $S$ is $\bdelta 2$.

$(3 \Rightarrow 4)$. For each $x\in X$, the singleton $\{x\}$ is in $\bsigma 2(X)$ by assumption, hence there are open sets $U_x$ and $V_x$ such that $\{x\}=U_x\setminus V_x$. Then $\Delta_X = \bigcup_{x\in X}\big[ \big(U_x\setminus V_x\big)\times\big(U_x\setminus V_x\big)\big]$ is in $\bsigma 2(X\times X)$. It was shown in \cite{debrecht2013} that the diagonal of every countably based $T_0$-space is $\bpi 2$, therefore $\Delta_X$ is in $\bdelta 2(X\times X)$.

$(4 \Rightarrow 1)$. Assume that $\Delta_X = \bigcup_{i\in\omega}U_i\setminus V_i$ for $U_i,V_i$ open in $X\times X$. Let $x$ be any element of $X$. Then there is some $i\in\omega$ such that $\langle x,x\rangle\in U_i\setminus V_i$. Let $U$ be an open neighborhood of $x$ such that $\langle x,x\rangle \in U\times U \subseteq U_i$. Fix any $y\in U$ distinct from $x$. Clearly, $\langle x,y\rangle\in U\times U\subseteq U_i$, hence $\langle x,y\rangle \in V_i$ because $\langle x,y\rangle\not\in \Delta_X$. Let $V$ and $W$ be open subsets of $X$ such that $\langle x,y\rangle \in V\times W\subseteq V_i$. Then $x\not\in W$ because otherwise we would have the contradiction $\langle x,x\rangle\in V_i$. Therefore, $W$ is a neighborhood of $y$ that does not contain $x$, hence $y$ is not in the closure of $\{x\}$. Thus, $\{x\}$ equals the intersection of $U$ with the closure of $\{x\}$, and it follows that $X$ is a $T_D$-space.
\end{proof}

The three spaces $S_2$, $S_1$, and $S_D$ defined in the introduction are the canonical countably based countable perfect $T_D$-spaces. The goal of this section is to prove that any non-empty countably based countable perfect $T_D$-space contains one of the above spaces.

\begin{lem}
If $X$ is a non-empty countably based countable perfect $T_D$-space, then either $X$ contains a perfect subspace homeomorphic to $S_D$ or else $X$ contains a non-empty perfect $T_1$-subspace.
\end{lem}
\begin{proof}
Define $Max(X)$ to be the subset of $X$ of elements that are maximal with respect to the specialization order. It is immediate that $Max(X)$ is a $T_1$-space.

First assume there is some $x_0 \in X$ such that there is no $y\in Max(X)$ with $x_0\leq y$. Then $x_0\not\in Max(X)$, so there is some $x_1\not=x_0$ with $x_0\leq x_1$. The assumption on $x_0$ implies $x_1\not\in Max(X)$, so there is $x_2\not=x_1$ with $x_0\leq x_1 \leq x_2$. Continuing in this way, we produce an infinite sequence $\{x_i\}_{i\in\omega}$ of distinct elements of $X$ with $x_i\leq x_j$ whenever $i\leq j$. Clearly $\{x_i\}_{i\in\omega}$, viewed as a subspace of $X$, is homeomorphic to $S_D$.

So if $X$ does not contain a copy of $S_D$, then every element of $X$ is below some element of $Max(X)$ with respect to the specialization order. This implies, in particular, that $Max(X)$ is non-empty. We show that $Max(X)$ is perfect as a subspace of $X$. Assume for a contradiction that there is $x\in Max(X)$ and open $V\subseteq X$ such that $\{x\}=V\cap Max(X)$. Since $X$ is a $T_D$-space, there is open $U\subseteq X$ such that $\{x\}=U\cap Cl(\{x\})$, where $Cl(\cdot)$ is the closure operator on $X$. Then $W=U\cap V$ is an open subset of $X$ containing $x$. Fix any $y\in W$. By assumption, there is some $y'\in Max(X)$ such that $y\leq y'$. Since $W$ is open, the definition of $\leq$ implies that $y'\in W$. Since $\{x\}=W\cap Max(X)$, it follows that $y' = x$ hence $y\leq x$. Therefore, $y\in Cl(\{x\})$ which implies $x=y$ because $\{x\}=W\cap Cl(\{x\})$. Since $y\in W$ was arbitrary, $\{x\}=W$ is an open subset of $X$, which contradicts $X$ being a perfect space. Therefore, $Max(X)$ is a non-empty perfect $T_1$-subspace of $X$.
\end{proof}

As a result of the above lemma, it only remains to consider the case of perfect $T_1$-spaces.

For any topological space $X$, open $U\subseteq X$, and $x\in X$, we write $x\lhd U$ if $x\in U$ and for every open $V$ containing $x$ and non-empty open $W\subseteq U$, the intersection $V\cap W$ is non-empty. In other words, $x\lhd U$ if and only if $x\in U$ and every neighborhood of $x$ is dense in the subspace $U$. Note that if $x\lhd U$ and $V\subseteq U$ is open and contains $x$, then $x\lhd V$. We define $D(X)$ to be the set of all $x\in X$ such that there is open $U\subseteq X$ with $x\lhd U$.

In the proofs of the next two lemmas, $X$ will be a countably based space with a basis $\{B_i\}_{i\in\omega}$ of open sets. We define $B(\cdot,\cdot)$ as in Definition~\ref{def:basis}.

\begin{lem}
If $X$ is a countably based countable perfect $T_1$-space and $D(X)$ has non-empty interior, then $X$ contains a perfect subspace homeomorphic to $S_1$.
\end{lem}
\begin{proof}
Choose any $x_0$ in the interior of $D(X)$ and let $U_0$ be an open subset of $X$ with $x_0\lhd U_0\subseteq D(X)$. Then $U_0$ is infinite because $X$ is perfect, so we can choose $x_1\in U_0$ distinct from $x_0$ and find open $U_1 \subseteq U_0$ with $x_1\lhd U_1$.

Let $n\geq 1$ and assume we have defined a sequence $x_0,\ldots,x_n\in X$ and open sets $U_0\supseteq \cdots \supseteq U_n$ with $x_i\lhd U_i\subseteq D(X)$. We choose $x_{n+1}\in X$ and open $U_{n+1}\subseteq U_n$ with $x_{n+1} \lhd U_{n+1}$ as follows. Define $V_i^n = U_i\cap B(x_i,n)$ for $0\leq i\leq n$, and let $V^n = V_0^n\cap \ldots \cap V_n^n$. Since $x_{n-1}\in V_{n-1}^n$ and $V_n^n\subseteq U_{n-1}$ is non-empty, $x_{n-1}\lhd U_{n-1}$ implies $V_{n-1}^n\cap V_n^n$ is non-empty. Continuing this argument inductively shows that $V^n$ is a non-empty open set. Thus $V^n$ is infinite, so there is $x_{n+1}\in V^n$ distinct from $x_i$ for $0\leq i\leq n$. Since $V^n\subseteq U_n \subseteq D(X)$, there is open $U_{n+1}\subseteq U_n$ with $x_{n+1}\lhd U_{n+1}\subseteq D(X)$.

Let $S=\{x_i\in X\,|\, i\in\omega\}$ be the subset of $X$ of the elements enumerated in the above construction. We claim that $S$ is homeomorphic to $S_1$. $S$ is infinite by construction, and the assumption that $X$ is a $T_1$-space implies that the subspace topology on $S$ contains the cofinite topology. Therefore, it suffices to show that every non-empty open subset of $S$ is cofinite. Let $U\subseteq S$ be non-empty open, so there is some $i\in\omega$ with $x_i\in U$. Let $m\geq i$ be large enough that $S\cap B(x_i,m)\subseteq U$. By the construction of $S$, $x_{n+1}\in V^n\subseteq B(x_i,n)\subseteq B(x_i,m)$ for all $n\geq m$. It follows that $x_{n+1}\in U$ for all $n \geq m$, hence $U$ is a cofinite subset of $S$.
\end{proof}

The final case to consider is when $X$ is a $T_1$-space and $X\setminus D(X)$ is dense in $X$.

\begin{lem}
If $X$ is a countably based countable perfect $T_1$-space and $X\setminus D(X)$ is dense in $X$, then $X$ contains a perfect subspace homeomorphic to $S_2$.
\end{lem}
\begin{proof}
Note that if $x \in X\setminus D(X)$ and $U$ is any open set containing $x$, then there exists non-empty open sets $V,W \subseteq U$ with $x\in V$ and $V\cap W=\emptyset$.

In the following, we denote the length of a sequence $\sigma\in 2^{<\omega}$ by $|\sigma|$. We associate each $\sigma\in 2^{<\omega}$ with an element $x_\sigma\in X\setminus D(X)$ and open set $U_\sigma\subseteq X$ containing $x_\sigma$ as follows. For the empty sequence $\varepsilon$ choose any $x_{\varepsilon}\in X\setminus D(X)$ and let $U_{\varepsilon}=X$. 

Next let $\sigma\in 2^{<\omega}$ be given and assume $x_\sigma\in X\setminus D(X)$ and $U_\sigma$ have been defined. Let $U,V\subseteq B(x_\sigma,|\sigma|)\cap U_\sigma$ be non-empty open sets such that $x_\sigma\in U$ and $U\cap V=\emptyset$. Since $V$ is non-empty and $X\setminus D(X)$ is dense, there exists some $y\in V\setminus D(X)$. Let $x_{\sigma\diamond 0} = x_{\sigma}$, $U_{\sigma\diamond 0} = U$, $x_{\sigma\diamond 1} = y$, and $U_{\sigma\diamond 0} = V$.

Let $S=\{x_\sigma\,|\, \sigma\in 2^{<\omega}\}$. A simple inductive argument shows that $U_\sigma\cap S$ is clopen in $S$ for each $\sigma\in 2^{<\omega}$. We show that $S$ is a perfect zero-dimensional $T_2$-space. Fix any $\sigma \in 2^{<\omega}$ and open $U\subseteq S$ containing $x_\sigma$. Let $n\in\omega$ be large enough that $B(x_\sigma,n)\cap S\subseteq U$. We can append a finite number of $0$'s to the end of $\sigma$ to obtain a sequence $\sigma'$ with $|\sigma'|\geq n$ and $x_{\sigma'}=x_{\sigma}$. Then $x_{\sigma'\diamond 1}\not=x_\sigma$ and $x_{\sigma'\diamond 1}\in B(x_\sigma,n)\cap S\subseteq U$. It follows that $\{x_{\sigma}\}$ is not open in $S$, so $S$ is a perfect space. Furthermore, $U_{\sigma'\diamond 0}\cap S$ is a clopen set containing $x_\sigma$ and contained in $U$, which implies that $S$ is a zero-dimensional $T_2$-space.

It follows that $S$ is a non-empty countable perfect metrizable space, hence $S$ is homeomorphic to $S_2$ (see Exercise 7.12 in \cite{kechris}).
\end{proof}

Combining the previous three lemmas with the observation that every subset of a countable $T_D$-space is $\bdelta 2$, we obtain the following.

\begin{thm}\label{thrm:canonical_perfectTD}
If $X$ is a non-empty countably based perfect $T_D$-space with countably many points, then $X$ contains a perfect $\bdelta 2$-subspace homeomorphic to either $S_D$, $S_1$, or $S_2$.
\qed
\end{thm}

\section{Completely Baire spaces}

A space is a \emph{Baire space} if every intersection of countably many dense open subsets is dense. A space is \emph{completely Baire} if every closed subspace is a Baire space. In this section we provide some characterizations of countably based $T_0$-spaces that are completely Baire. In particular, we show that a countably based $T_0$-space is completely Baire if and only if it does not contain a non-empty $\bpi 2$-subset which is a countable perfect $T_D$-space. 

The following theorem is a generalization of some results by W. Hurewicz \cite{hurewicz}.

\begin{thm}\label{thrm:completelyBaire_characterization}
The following are equivalent for a countably based $T_0$-space $X$:
\begin{enumerate}
\item
$X$ is completely Baire,
\item
$X$ does not contain a non-empty $\bpi 2$-subset which is a countable perfect $T_D$-space,
\item
$X$ does not contain a $\bpi 2$-subset homeomorphic to $S_D$, $S_1$, or $S_2$,
\item
Every $\bpi 2$-subspace of $X$ is a Baire space.
\item
Every $\bpi 2$-subspace of $X$ is a completely Baire space.
\end{enumerate}
\end{thm}
\begin{proof}
$(1 \Rightarrow 2)$. Assume there is some $Y\in \bpi 2(X)$ which is a countable perfect $T_D$-space. Then we can write $Y=\bigcap_{i\in\omega} (A_i\cup U_i)$, with $A_i\subseteq X$ closed and $U_i\subseteq X$ open. Let $C$ be the closure of $Y$ in $X$, and let $\{y_i\}_{i\in\omega}$ be an enumeration of the elements of $Y$.

Let $W_i = C\setminus Cl_C(\{y_i\})$, where $Cl_C(\cdot)$ is the closure operator in the subspace $C$ of $X$. We claim that $W_i$ is dense and open in the subspace $C$. For a contradiction, assume $U$ is a non-empty relatively open subset of $C$ and $U\cap W_i=\emptyset$. Since $Y$ is dense in $C$, there must be some $y\in U\cap Y$. Then $y\not\in W_i$, so $y\in Cl_C(\{y_i\})$ hence $y_i\in U$. $Y$ is a $T_D$-space, so there is some $V$ relatively open in $C$ such that $\{y_i\}=V\cap Cl_C(\{y_i\})\cap Y$. Then $\{y_i\}=U\cap V\cap Y$, contradicting $Y$ being a perfect space. Therefore, $W_i$ is a dense open subset of the subspace $C$.

Next define $V_i=(Int_C(A_i)\cup U_i)\cap C$, where $Int_C(\cdot)$ is the interior operator in the subspace $C$. We claim that $V_i$ is also dense and open in the subspace $C$. Let $U$ be a non-empty open set in the subspace $C$. If $U\setminus A_i$ is non-empty, then there is $y\in (U\setminus A_i)\cap Y$ (since $Y$ is dense in $C$), hence $y\in U\cap (U_i\cap Y)$ so $y\in U\cap V_i$. On the other hand, if $U\setminus A_i$ is empty then $U\subseteq Int_C(A_i)$, so $U\cap V_i$ is non-empty.

Let $S=(\bigcap_{i\in\omega}W_i)\cap(\bigcap_{i\in\omega}V_i)$. The choice of $W_i$ imply that $S\cap Y=\emptyset$, and the choice of $V_i$ imply $S\subseteq Y$. Therefore, $S$ is empty. But $S$ is a countable intersection of dense open subsets of $C$, so $C$ is not a Baire space. Therefore, $X$ is not completely Baire.

$(2 \Rightarrow 1)$. We show that if $X$ is not a completely Baire space then there is a countable perfect $T_D$-space $Y\in \bpi 2(X)$. For any closed set $C\subseteq X$, any $\bpi 2$-subset of $C$ is a $\bpi 2$ subset of $X$, so it suffices to only consider the case that $X$ is not a Baire space. So assume there are dense open $U_i\subseteq X$ such that $\bigcap_{i\in\omega}U_i$ is not dense in $X$. Then there is non-empty open $U\subseteq X$ such that $U\cap \bigcap_{i\in\omega}U_i$ is empty. Clearly $U_i\cap U$ is dense and open in the subspace $U$, and again any $\bpi 2$-subset of $U$ is a $\bpi 2$ subset of $X$. We can therefore assume without loss of generality that $\bigcap_{i\in\omega}U_i$ is empty in $X$.

Fix a countable basis $\{B_i\}_{i\in\omega}$ of open subsets of $X$ and define $B(\cdot,\cdot)$ as in Definition~\ref{def:basis}. Fix a function $f\colon\omega\to\omega$ such that $f(n)\leq n$ and $f^{-1}(\{n\})$ is infinite for each $n\in\omega$.

Choose any $x_0\in X$. For $n\geq 0$, let $V_n = B(x_{f(n)},n)\setminus Cl(\{x_{f(n)}\})$, where $Cl(\cdot)$ is the closure operator on $X$. If $V_n$ was empty then every open set intersecting $B(x_{f(n)},n)$ would contain $x_{f(n)}$, which would imply $x_{f(n)}$ is in the intersection of the dense open sets $U_i$, a contradiction. Therefore, $V_n$ is a non-empty open set, so $V_n\cap U_0$ is non-empty open, and continuing inductively we have that $V_n\cap U_0\cap \cdots\cap U_n$ is non-empty. Let $x_{n+1}$ be any element of this intersection.

Define $Y=\{x_n\,|\, n\in\omega\}$. For each $x\in Y$ and $n\in\omega$, there is some $y\in Y\cap B(x,n)$ that is not in the closure of $\{x\}$. Thus every open neighborhood of $x$ contains a point in $Y$ that is distinct from $x$, hence $Y$ is a perfect space. To see that $Y$ is a $T_D$-space, let $x\in Y$ be given and let $m\in\omega$ be the smallest number such that $x\not\in U_0\cap \ldots\cap U_m$. Then for all $n> m$, our choice of $x_n$ guarantees that it is in an open set that does not contain $x$, hence $x_n$ is not in the closure of $\{x\}$. Thus $Cl(\{x\})\cap Y$ is finite, hence there is an open neighborhood $U$ of $x$ such that $Cl(\{x\})\cap U \cap Y = \{x\}$. It follows that $Y$ is a $T_D$-space.

It only remains to show that $Y\in\bpi 2(X)$. For $n\in\omega$, define 
\[A_n = \{x_0,\ldots,x_n\}\cup \bigcap_{i\leq n} U_i.\]
Since finite subsets of countably based $T_0$-spaces are $\bpi 2$, it is easy to see that $A_n\in\bpi 2(X)$. Then $A=\bigcap_{n\in\omega}A_n$ is also in $\bpi 2(X)$. By our construction, $Y$ is a subset of $A$. Now let $x\in A$ be given. Then there is some $n\in\omega$ such that $x\not\in U_0\cap \cdots\cap U_n$. Since $x\in A_n$, it follows that $x\in\{x_0,\ldots,x_n\}\subseteq Y$. Therefore, $Y=A$ is a countable perfect $T_D$-space in $\bpi 2(X)$.

$(2 \Leftrightarrow 3)$. Immediate from Theorem \ref{thrm:canonical_perfectTD} and the fact that $S_D$, $S_1$, and $S_2$ are all countable perfect $T_D$-spaces.

$(4 \Leftrightarrow 5)$. The implication $(5\Rightarrow 4)$ is trivial, and the implication $(4 \Rightarrow 5)$ holds because if $Y\in \bpi 2(X)$ and $Z\subseteq Y$ is closed, then $Z\in \bpi 2(X)$ and is therefore a Baire space.

$(5 \Leftrightarrow 1)$. The implication $(5\Rightarrow 1)$ is trivial. For the converse, if $Y\in \bpi 2(X)$ is not completely Baire, then using the equivalence $(1 \Leftrightarrow 2)$ already shown above, there is a non-empty countable perfect $T_D$-space $Z\in \bpi 2(Y)$. But then $Z\in\bpi 2(X)$, hence $X$ is not completely Baire.
\end{proof}

The reader should note that in the presence of stronger separation axioms we can reduce the upper bound on the Borel complexity of the countable perfect $T_D$-subspace mentioned in Theorem \ref{thrm:completelyBaire_characterization}. For example, if $X$ is a countably based $T_D$-space then every countable subset of $X$ is a $\bsigma 2$-set, hence $X$ is completely Baire if and only if it does not contain a $\bdelta 2$-subspace homeomorphic to $S_D$, $S_1$, or $S_2$. Hurewicz's original theorem shows that a separable metrizable space is completely Baire if and only if it does not contain a closed subspace homeomorphic to $S_2$.

We can also give a slightly stronger characterization for countably based sober spaces. Recall that a non-empty closed set is \emph{irreducible} if it is not the union of two proper closed subsets. A space is \emph{sober} if and only if every irreducible closed set equals the closure of a unique point. Every sober space is a $T_0$-space, and every $T_2$-space is sober. Sobriety is incomparable with the $T_1$ separation axiom. Neither $S_D$ nor $S_1$ is sober, because in both cases the entire space is an irreducible closed set which is not the closure of any singleton. Theorem~\ref{thrm:sober_characterization} below shows that these two spaces are the canonical examples of non-sober countably based spaces.

\begin{lem}\label{lem:pi2subspace_of_sober}
Every $\bpi 2$-subspace of a sober space is sober.
\end{lem}
\begin{proof}
Assume $Y$ is a sober space, and assume for a contradiction that there is some $Z\in\bpi 2(Y)$ that is not sober. Let $X\subseteq Z$ be an irreducible closed subset of $Z$ that is not equal to the closure of any singleton. Clearly, $X\in\bpi 2(Y)$. 

Let $Cl(\cdot)$ be the closure operator for $Y$ and define $C = Cl(X)$. Assume $A_1$ and $A_2$ are closed subsets of $Y$ such that $C=A_1\cup A_2$. Then $X=(A_1\cap X)\cup(A_2\cap X)$, so the irreducibility of $X$ implies $X\subseteq A_1$ or $X\subseteq A_2$. Since $C$ is the closure of $X$ in $Y$, it follows that $C=A_1$ or $C=A_2$, hence $C$ is an irreducible closed subset of $Y$. Therefore, the sobriety of $Y$ implies $C=Cl(\{x\})$ for some $x\in Y$. Clearly $x\not\in X$ because by assumption $X$ is not the closure of any singleton.

Since $Y\setminus X$ is a $\bsigma 2$-subset of $Y$ containing $x$, there must be some open $U\subseteq Y$ satisfying $x\in U\cap Cl(\{x\}) \subseteq Y\setminus X$, but the existence of such an open set $U$ contradicts $Cl(\{x\}) = Cl(X)$.
\end{proof}

\begin{thm}\label{thrm:sober_characterization}
A countably based $T_0$-space is sober if and only if it does not contain a $\bpi 2$-subspace homeomorphic to $S_D$ or $S_1$.
\end{thm}
\begin{proof}
If $X$ is a sober space, then it follows from Lemma~\ref{lem:pi2subspace_of_sober} that $X$ does not contain a $\bpi 2$-subspace homeomorphic to $S_D$ or $S_1$.

For the converse, assume $X$ is a non-sober countably based $T_0$-space. Since every closed subset of $X$ is a $\bpi 2$-subset, we can assume without loss of generality that $X$ itself is an irreducible closed set that does not equal the closure of any singleton subset. Furthermore, we can assume that $X\subseteq Y$ for some quasi-Polish space $Y$ (in particular, we can take $Y=\cal P(\omega)$, which is quasi-Polish and well-known to be universal for countably based $T_0$-spaces). Note that every quasi-Polish space is sober (Corollary~39 of \cite{debrecht2013}).

Let $Cl(\cdot)$ be the closure operator for $Y$. Using the same argument as in the proof of Lemma~\ref{lem:pi2subspace_of_sober}, there is some $x\in Y\setminus X$ satisfying $Cl(\{x\}) = Cl(X)$.

Fix a countable basis $\{B_i\}_{i\in\omega}$ of open subsets of $Y$ and define $B(\cdot,\cdot)$ as in Definition~\ref{def:basis}. Note that every open neighborhood $V$ of $x$ satisfies $V\cap X\not=\emptyset$ because $Cl(\{x\}) = Cl(X)$. Furthermore, the irreducibility of $X$ implies that if $\{U_i\}_{i\in F}$ is a finite collection of open subsets of $Y$ that intersect $X$, then $\bigcap_{i\in F} U_i$ also intersects $X$.

We inductively define an infinite sequence $\{x_n\}_{n\in\omega}$ of distinct elements of $X$. Define $V_0 =B(x,0)$ and let $x_0$ be any point in $V_0\cap X$. For $n\geq 0$, let $U_{n+1} \subseteq V_n$ be an open subset of $Y$ containing $x$ that does not contain any $x_i$ for $i\leq n$. Such an open set exists because each $x_i$ is distinct from $x$ and contained in $Cl(\{x\})$. Let $V_{n+1} = U_{n+1} \cap B(x,n+1) \cap \bigcap_{i\leq n}B(x_i,n+1)$. Choose $x_{n+1}$ to be any element of $V_{n+1} \cap X$, which is possible because of the observation in the previous paragraph.

Then $A = \{x_n \,|\, n\in\omega\}$ is an infinite subset of $X$, and $\{V_n\}_{n\in\omega}$ is a decreasing sequence of neighborhoods of $x$ satisfying $\{x\} = Cl(\{x\}) \cap \bigcap_{n\in\omega}V_n$. In particular,  $\bigcap_{n\in\omega}V_n$ has empty intersection with $X$. As in the proof of Theorem~\ref{thrm:completelyBaire_characterization}, we can define
\[A_n = \{x_0,\ldots,x_n\}\cup \bigcap_{i\leq n} V_i\]
and obtain that $A = X\cap \bigcap_{n\in\omega} A_n$ is in $\bpi 2(X)$. 

By our construction, $Cl(\{x_n\})\cap A$ is finite for each $n\in\omega$ because $x_n\not\in V_{n+1}$ but each $x_i \in V_{n+1}$ for $i>n$. Therefore, $A$ is a countable $T_D$-space. Furthermore, $A$ is a perfect space because for any $x_i$ and $n\geq i$ we have $x_{n+1}\in B(x_i,n+1)$, hence $\{x_i\}$ is not open in $A$. For a similar reason, $A$ does not contain any infinite $T_2$-subspaces because for any $i,j\in\omega$ and $n\geq i,j$ we have that all but finitely many elements of $A$ are in $B(x_i, n+1)\cap B(x_j,n+1)$.

By applying Theorem~\ref{thrm:canonical_perfectTD}, there is $S\in \bdelta 2(A)$ that is homeomorphic to either $S_D$, $S_1$, or $S_2$. As we have just shown that $A$ does not contain any infinite $T_2$-spaces, $S$ must be homeomorphic to either $S_D$ or $S_1$. Clearly, $S\in\bpi 2(X)$ because $S\in \bdelta 2(A)$ and $A\in \bpi 2(X)$.
\end{proof}

Theorems \ref{thrm:completelyBaire_characterization} and \ref{thrm:sober_characterization} imply the following.

\begin{thm}\label{thrm:sober_not_completelyBaire_then_S2}
A countably based sober space is completely Baire if and only if it does not contain a $\bpi 2$-subspace homeomorphic to $S_2$.
\qed
\end{thm}

\section{\texorpdfstring{$\bpi 3$}{Pi03}-spaces}

We will call a space a $\bpi 3$-space if and only if it is homeomorphic to a $\bpi 3$-subset of a quasi-Polish space. In Section 6.1 of \cite{debrecht2013} it was shown that a countably based $T_0$-space admits a bicomplete quasi-metric if and only if it is a $\bpi 3$-space.

The first goal of this section is to show that a $\bpi 3$-space is completely Baire if and only if it is quasi-Polish. It was shown in \cite{debrecht2013} (Corollary 52) that quasi-Polish spaces are Baire spaces, and since every $\bpi 2$-subspace of a quasi-Polish space is quasi-Polish, it follows that every quasi-Polish space is completely Baire. Theorem \ref{thrm:Pi3_completelyBaire} below provides a converse for the class of $\bpi 3$-spaces. Combining our previous results, this implies that a $\bpi 3$-space is either quasi-Polish, or else it contains a $\bpi 2$-subset homeomorphic to $S_D$, $S_1$, or $S_2$. This provides a clear separation between $\bpi 3$ and $\bpi 2$ (i.e., quasi-Polish) spaces.

The following lemma is a generalization of the Baire category theorem, and has been investigated by R. Heckmann \cite{heckmann} and Becher and Grigorieff \cite{becher_grigorieff} in its dual form (i.e., countable intersections of dense $\bpi 2$-sets are dense).

\begin{lem}
If $X$ is a Baire space and $\{A_i\}_{i\in\omega}$ is a countable collection of sets in $\bsigma 2(X)$ satisfying $X= \bigcup_{i\in\omega} A_i$, then some $A_i$ has non-empty interior with respect to $X$.
\end{lem}
\begin{proof}
Let $A_i = \bigcup_{j\in\omega} (U^i_j \cap C^i_j)$, with $U^i_j$ open and $C^i_j$ closed in $X$. Define $F^i_j = Cl(U^i_j\cap C^i_j)\subseteq Cl(U^i_j)\cap C^i_j$, where $Cl(\cdot)$ is the closure operator on $X$. Since $X$ is a Baire space and $X = \bigcup_{i,j\in\omega} F^i_j$, there is non-empty open $U\subseteq F^i_j$ for some $i,j\in\omega$. The open set $W=U\cap U^i_j$ is non-empty because $U$ intersects the closure of $U^i_j$, and  clearly $W\subseteq U^i_j\cap C^i_j \subseteq A_i$.
\end{proof}

\begin{lem}
Let $X$ be quasi-Polish, $Y \subseteq X$ completely Baire, and assume $A\in\bpi 2(X)$ is such that $A\cap Y=\emptyset$. Then there is $C\in\bsigma 2(X)$ separating $A$ from $Y$ (i.e., $A\subseteq C$ and $C\cap Y=\emptyset$).
\end{lem}
\begin{proof}
Assume for a contradiction that there is no $\bsigma 2$-set separating $A$ from $Y$. Let $\{B_i\}_{i\in\omega}$ be a basis for $X$. Define
\[ U = \bigcup\{B_i\,|\, \exists C_i\in\bsigma 2(X): B_i\cap A\subseteq C_i \mbox{ \& } C_i \cap Y=\emptyset\}.\]

Then $U$ is open, and there is $C\in\bsigma 2(X)$ separating $A\cap U$ from $Y$ (just take $C$ to be the countable union of the $C_i$'s). Therefore, if we set $A'=A\setminus U$, we must have $Z=Cl(A')\cap Y\not=\emptyset$. Clearly $A'\in\bpi 2(X)$, so we can write $A' = \bigcap_{i\in\omega}D_i$ with $D_i\in\bdelta 2(X)$. Since $Z$ is a Baire space and $Z\subseteq \bigcup_{i\in\omega}X\setminus D_i$, there is $i,j\in\omega$ such that $B_i\cap Z\not=\emptyset$ and $B_i\cap Z\subseteq X\setminus D_j$. Clearly $B_i\cap A'\not=\emptyset$ because $B_i$ intersects the closure of $A'$.

Now let $C' = B_i\cap Cl(A') \cap D_j$. Then $C'\in\bsigma 2(X)$, $B_i\cap A'\subseteq C'$, and $Y\cap C' = \emptyset$. Thus $C\cup C'$ is a $\bsigma 2$-set separating $B_i\cap A$ from $Y$. But then we must have $B_i\subseteq U$, which contradicts $B_i\cap A'\not=\emptyset$.
\end{proof}

\begin{thm}\label{thrm:Pi3_completelyBaire}
If $X$ is quasi-Polish and $Y\in\bpi 3(X)$ is completely Baire, then $Y$ is quasi-Polish.
\end{thm}
\begin{proof}
We can write $X\setminus Y = \bigcup_{i\in\omega} A_i$ with $A_i\in\bpi 2(X)$. From the previous lemma there is $C_i\in\bsigma 2(X)$ separating $A_i$ from $Y$. Therefore, $X\setminus Y = \bigcup_{i\in\omega}C_i$ is $\bsigma 2$, hence $Y\in\bpi2 (X)$. Every $\bpi 2$-subspace of a quasi-Polish space is quasi-Polish (Theorem 23 of \cite{debrecht2013}), hence $Y$ is quasi-Polish.
\end{proof}

We now move on to the next goal of this section, which is to characterize which \emph{countable} spaces are $\bpi 3$-spaces (Theorem \ref{thrm:CharacterizationOfCountableDelta3} below). Since every countable subset of a countably based $T_0$-space is $\bsigma 3$, this amounts to determining when a countable subset of a quasi-Polish space is a $\bdelta 3$-set. In particular, we will see that every countably based countable $T_D$-space is a $\bpi 3$-space.

\begin{lem}\label{lem:ConditionForCountableBeingDelta3}
Assume $Y$ is a countably based $T_0$-space and $X\subseteq Y$ is countable. If for every non-empty $A\in \bpi 2(X)$ there is a finite non-empty $F\in\bdelta 2(A)$, then $X\in\bdelta 3(Y)$.
\end{lem}
\begin{proof}
For each ordinal $\alpha$, we inductively define $X^{\alpha}$ as follows:
\begin{itemize}
\item
$X^0 = X$,
\item
$X^{\alpha+1} = X^{\alpha}\setminus \{x\in X^{\alpha}\,|\, \{x\} \mbox{ is locally closed in } X^{\alpha}\}$,
\item
$X^{\alpha} = \bigcap_{\beta<\alpha} X^{\beta}$ when $\alpha$ is a limit ordinal.
\end{itemize}

Since $X$ is countable there is some ordinal $\alpha<\omega_1$ such that $X^{\alpha}=X^{\alpha+1}$. We define $\ell(X)$ to be the least such ordinal. Using again the fact that $X$ is countable, it is straight forward to show that $X^{\alpha}\in\bpi 2(X)$ for each $\alpha<\ell(X)$. Thus our assumption on $X$ implies that if $X^{\alpha}$ is not empty, then there is a finite non-empty $F\in \bdelta 2(X^{\alpha})$. It follows that $\{x\}$ is locally closed in $X^{\alpha}$ for each $x\in F$, hence $X^{\alpha}\not=X^{\alpha+1}$. Therefore, $X^{\ell(X)}=\emptyset$.

The lemma is trivial if $X$ is finite, so fix an infinite enumeration $x_0,x_1,\ldots$ of $X$ without repetitions. Since $X^{\ell(X)}=\emptyset$, for each $i\in \omega$ there is a countable ordinal $\alpha_i<\ell(X)$ such that $x_i\in X^{\alpha_i}\setminus X^{\alpha_i+1}$. Choose an open subset $U_i$ of $Y$ such that $Cl(\{x_i\})\cap U_i\cap X^{\alpha_i} = \{x_i\}$ (here and in the following, $Cl(\cdot)$ is the closure operator for $Y$).

For each $i\in\omega$, define $A_i = Cl(\{x_i\})\cap U_i$. Then $A_i\in\bdelta 2(Y)$ and $x_i \in A_i$.

Next, for each $i\in\omega$, let $\{V^i_j\}_{j\in\omega}$ be a decreasing sequence of open subsets of $Y$ such that $\{x_i\} = Cl(\{x_i\})\cap \bigcap_{j\in\omega}V^i_j$, and $x_k\not\in V^i_j$ whenever $k\leq j$ and $x_i\not\in Cl(\{x_k\})$.

Define $W_j = \bigcup_{i\in\omega} A_i\cap V^i_j$. Then $W=\bigcap_{j\in\omega} W_j$ is in $\bpi 3(Y)$, and $X\subseteq W$ is clear from the construction.

Next, let $y\in W$ be fixed. The set of ordinals $\{\alpha_i\,|\, y\in A_i\}$ is non-empty, so let $\alpha$ be its minimal element. Then there is some $k\in\omega$ satisfying $\alpha_k=\alpha$ and $y\in A_k$ (it actually turns out that $k$ is uniquely determined). 

Assume for a contradiction that there is $j\geq k$ and $i\not=k$ such that $y\in A_i\cap V^i_j$. Then $x_k\in V^i_j$ because $V^i_j$ is an open set containing $y$ and $y\in Cl(\{x_k\})$. Thus, $k \leq j$ together with our definition of $V^i_j$ implies $x_i\in Cl(\{x_k\})$. We also have $x_i\in U_k$ because $y\in U_k$ and $y\in Cl(\{x_i\})$. Since $Cl(\{x_k\})\cap U_k \cap X^{\alpha_k}=\{x_k\}$, we must have $x_i\not\in X^{\alpha_k}$. But then $y\in A_i$ and $\alpha_i < \alpha_k$, contradicting our choice of $\alpha$.

Since $y\in \bigcap_{j\in\omega} W_j$, the above argument implies that $y\in A_k\cap V^k_j$ for all $j\geq k$. It follows that $y\in Cl(\{x_k\})\cap V^k_j$ for all $j\in \omega$ because $A_k\subseteq Cl(\{x_k\})$ and because the sequence $\{V^k_j\}_{j\in\omega}$ is decreasing. Our choice of $V^k_j$ implies $y=x_k$, and since $y\in W$ was arbitrary, we obtain $W\subseteq X$.

Therefore, $X=W\in \bpi 3(Y)$. As every countable subset of a countably based space is a $\bsigma 3$-set, it follows that $X\in\bdelta 3(Y)$.
\end{proof}

The use of transfinite ordinals in the above proof  might seem excessive. However, the following example suggests that it is not avoidable.

Let $\omega^{<n}$ be the set of sequences of natural numbers of length less than $n$. Give $\omega^{<n}$ the topology generated by subbasic open sets of the form $B_{\sigma} = \omega^{<n} \setminus\{\sigma'\in \omega^{<n} \,|\, \sigma \preceq \sigma'\}$, where $\sigma$ varies over elements of $\omega^{<n}$ and $\preceq$ is the prefix relation. The specialization order on $\omega^{<n}$ is simply $\succeq$. Then $\{\sigma\}$ is locally closed in $\omega^{<n}$ if and only if the length of $\sigma$ equals $n-1$. Therefore, $\ell(\omega^{<n})=n$. If we take $X$ to be the disjoint union of the sequence of spaces $\{\omega^{<n}\}_{n\in\omega}$, then $\ell(X)=\omega$.

Every finite subset of a $T_D$-space is a $\bdelta 2$-set, so we immediately obtain the following corollary of Lemma \ref{lem:ConditionForCountableBeingDelta3}.

\begin{cor}\label{cor:countableTD_delta3}
If $Y$ is a countably based $T_0$-space and $X\subseteq Y$ is a countable $T_D$-space, then $X\in \bdelta 3(Y)$.
\qed
\end{cor}

If $Y$ is quasi-Polish, then the converse of Lemma \ref{lem:ConditionForCountableBeingDelta3} holds as well, providing a complete characterization of countable $\bpi 3$-spaces.

\begin{thm}\label{thrm:CharacterizationOfCountableDelta3}
Assume $Y$ is quasi-Polish and $X\subseteq Y$ is countable. Then the following are equivalent:
\begin{enumerate}
\item
$X\in \bdelta 3(Y)$,
\item
Every non-empty $A\in \bpi 2(X)$ contains a finite non-empty $F\in\bdelta 2(A)$,
\item
For every non-empty $A\subseteq X$ there is $x\in A$ such that $\{x \}$ is locally closed in $A$.
 \end{enumerate}
\end{thm}
\begin{proof}
The implication $(3 \Rightarrow 2)$ is trivial and the implication $(2 \Rightarrow 1)$ follows from Lemma~\ref{lem:ConditionForCountableBeingDelta3}.

We only provide a sketch of the proof for the implication $(1 \Rightarrow 3)$, and the reader should consult Section 6.1 of \cite{debrecht2013} for background on the properties of bicomplete quasi-metrics that we use here.

If $X\in \bdelta 3(Y)$, then by Theorem 32 of \cite{debrecht2013} there is a quasi-metric $d$ compatible with the topology on $X$ such that the induced metric space $(X,\widehat{d})$ is Polish. Since $(X,\widehat{d})$ is a countable Polish space it is scattered, which means it does not contain a non-empty perfect subspace (see \cite{kechris} or Section 12 of \cite{debrecht2013}). This implies that for any non-empty $A \subseteq X$ there is $x\in A$ such that $\{x\}$ is open in $(A,\widehat{d})$. By Theorem 14 of \cite{debrecht2013}, every open subset of $(A,\widehat{d})$ is a $\bsigma 2$-subset of $(A,d)$, hence $\{x\}$ is a $\bsigma 2$-subset of $(A,d)$. Therefore, $\{x\} = U\setminus V$ for some pair of sets $U,V\subseteq A$ that are open in $(A,d)$, hence $\{x\}$ is locally closed within the subspace $A$ of $X$.
\end{proof}

\section{Countable spaces}

So far we have not had much to say about the space $S_0$, but it plays the central role in this section where we can finally characterize all of the countably based countable $T_0$-spaces that are not quasi-Polish. Note that the specialization order on $S_0$ is the inverse of the prefix order on $\omega^{<\omega}$.

In this section, we will prove that a countably based completely Baire $T_0$-space with countably many points is either quasi-Polish or else contains $S_0$ as a $\bpi 2$-subspace. Combined with our previous results (Theorems \ref{thrm:completelyBaire_characterization} and \ref{thrm:Pi3_completelyBaire}), this implies that a countably based countable $T_0$-space is either quasi-Polish or else it contains a $\bpi 2$-subspace homeomorphic to $S_2$, $S_1$, $S_D$, or $S_0$.

As stated in the introduction, a subbasis for the closed subsets of $S_0$ is given by sets of the form $\uparrow\!p = \{q\in\omega^{<\omega}\,|\, p\preceq q\}$ for $p\in\omega^{<\omega}$. By taking finite unions we obtain a basis for the closed subsets of $S_0$. In particular, if $F\subseteq S_0$ is finite, then its closure $Cl(F)$ in $S_0$ equals the finite union $\bigcup_{p\in F} \uparrow\! p$. Thus every closed subset of $S_0$ is equal to an intersection of the form $\bigcap_{i\in I} Cl(F_i)$, where $\{F_i\}_{i\in I}$ is a collection of finite subsets of $S_0$. It follows that if $A\subseteq S_0$ is closed and $x\not\in A$, then there is finite $F\subseteq S_0$ such that $A \subseteq Cl(F)$ and $x\not\in Cl(F)$. Because of this fact, proofs concerning the topology of $S_0$ tend to focus on closures of finite subsets, but it is important to note that not all closed subsets of $S_0$ are of this form. In fact, the topology of $S_0$ is uncountable, which the author only became aware of after discussions with Victor Selivanov concerning the space $S_0$.

\begin{prop}\label{prop:S0_topology_uncountable}
The topology on $S_0$ is uncountable.
\end{prop}
\begin{proof}
We construct an injection that maps each $p\in\baire$ to a closed subset $A_p$ of $S_0$. Given $p\in\baire$, define $F^p_n = \{ 0^m \diamond (p(m)+1) \,|\, m\leq n\}\cup\{ 0^{n+1}\}$ for each $n\in\omega$, where $0^m$ denotes the string of length $m$ consisting only of zeros (this is the empty string $\varepsilon$ for $m=0$). Thus 
\begin{eqnarray*}
F^p_0 &=& \{ (p(0)+1), \, 0\},\\
F^p_1 &=& \{ (p(0)+1), \, 0\diamond(p(1)+1), \, 00\},\\
F^p_2 &=& \{ (p(0)+1), \, 0\diamond(p(1)+1), \, 00\diamond(p(2)+1), \, 000\},
\end{eqnarray*}
and so on. Finally, define $A_p = \bigcap_{n\in\omega} Cl(F^p_n)$. Then $A_p = \bigcup_{n\in\omega} \uparrow \! \big(0^n\diamond(p(n)+1)\big)$. In particular, $A_p\not= Cl(F)$ for any finite $F\subseteq S_0$.

Finally, we show that the mapping $p\mapsto A_p$ is an injection. Fix any distinct pair $p,q\in\baire$ and $n\in\omega$ with $p(n)\not=q(n)$. Then $0^n\diamond (p(n)+1)$ is in $A_p$ but not in $A_q$ because the only sequence in $A_q$ consisting of a sequence of $n$ zeros followed by a single non-zero number is $0^n\diamond (q(n)+1)$. 
\end{proof}


From Theorem \ref{thrm:CharacterizationOfCountableDelta3} above, we know that $S_0$ is not quasi-Polish, and not even a $\bpi 3$-space, because there is no $x\in S_0$ for which $\{x\}$ is locally closed. Indeed, if $x \in U\cap A$ with $U\subseteq S_0$ open and $A\subseteq S_0$ closed, then from our discussion of the topology of $S_0$ above, there is finite $F\subseteq S_0$ such that the complement of $U$ is contained in $Cl(F)$ and $x\not\in Cl(F)$. Since $x\not\in Cl(F)$, any immediate successor (with respect to $\preceq$) of $x$ that is in $Cl(F)$ must actually be in the finite set $F$, which implies that all but finitely many immediate successors of $x$ are in $U$. All of the immediate successors of $x$ are in $A$, hence infinitely many immediate successors of $x$ are in $U\cap A$. Therefore, $\{x\}$ is not locally closed.

However, we can show that $S_0$ is a completely Baire space.

\begin{lem}\label{lem:S0_closed_maximal_discrete}
If $A\subseteq S_0$ is closed, then the subset $D\subseteq A$ of elements that are maximal with respect to the specialization order $\leq$ is a discrete subspace. Furthermore,
\begin{align*}
  A = Cl(D) = \{ y\in S_0 \,|\, (\exists z\in D)\,y\leq z\}.
\end{align*}
\end{lem}
\begin{proof}
We first prove the last statement of the lemma. Let $C = \{ y\in S_0 \,|\, (\exists z\in D)\,y\leq z\}$. It is clear that $Cl(D) \subseteq A$. Furthermore, since the prefix order on $\omega^{<\omega}$ is well-founded, every element of $A$ is below a maximal element of $A$ with respect to $\leq$, hence $A \subseteq C$. Finally, for every $y\in C$ there is $z\in D$ such that $y\in Cl(\{z\})$, hence $C \subseteq Cl(D)$.

Next we show that every singleton subspace of $D$ is open in the subspace topology on $D$. Let $x$ be any element of $D$. If $x$ is the maximal element of $S_0$ (the empty string $\varepsilon$), then $D=\{x\}$ and the proof is complete. Otherwise, $x=\sigma\diamond n$ for some $\sigma\in\omega^{<\omega}$ and $n\in\omega$. Maximality of $x$ in $A$ implies $\sigma\not\in A$, so there is finite $F\subseteq S_0$ such that $A\subseteq Cl(F)$ and $\sigma\not\in Cl(F)$. Note that $x$ is in $F$ and in fact $x$ is a maximal element of $F$ because $x$ is the immediate successor of $\sigma$ with respect to the prefix relation. Let $U$ be the complement of $Cl(F\setminus\{x\})$. Maximality of $x$ in $F$ guarantees that $x\in U$. Furthermore, $U\cap D \subseteq Cl(F)\setminus Cl(F\setminus \{x\})$, which implies every $z\in U\cap D$ satisfies $z\leq x$ and thus $z=x$ by maximality of $z$ in $D$. It follows that $U\cap D = \{x\}$, and therefore $D$ is a discrete subspace.
\end{proof}

\begin{thm}
$S_0$ is completely Baire.
\end{thm}
\begin{proof}
Let $A$ be a non-empty closed subspace of $S_0$. From the previous lemma, the subset $D$ of maximal elements of $A$ is a discrete subspace and is non-empty because $A=Cl(D)$.

Let $V$ be a dense (relatively) open subset of $A$. Given any $x\in D$, there is a (relatively) open subset $U\subseteq A$ such that $U\cap D = \{x\}$. Then $V\cap U \not=\emptyset$ because $V$ is dense in $A$, hence $V\cap U\cap D \not=\emptyset$ because $D$ is dense in $A$. Therefore, $x$ is in $V$. 

It follows that any dense open subset of $A$ must contain all of $D$, which implies $A$ is a Baire space.
\end{proof}


We next work towards showing that $S_0$ is in a sense canonical among the countable completely Baire spaces that are not quasi-Polish spaces. As before, in the following proofs we will denote the specialization order by $\leq$. 

\begin{lem}\label{lem:containsS0}
Let $X$ be a non-empty countably based $T_0$-space such that:
\begin{enumerate}
\item
$X$ has countably many points,
\item
$X$ is completely Baire,
\item
Every non-empty $\bdelta 2$-subset of $X$ is infinite.
\end{enumerate}
Then $X$ contains a $\bpi 2$-subspace homeomorphic to $S_0$.
\end{lem}
\begin{proof}
Fix a countable basis $\{B_i\}_{i\in\omega}$ of open subsets of $X$ and define $B(\cdot,\cdot)$ as in Definition~\ref{def:basis}. Fix a bijection $r\colon\omega^{<\omega}\to\omega$ such that $r(\varepsilon)=0$, $r(\sigma)\leq r(\sigma\diamond n)$, and $r(\sigma\diamond n)\leq r(\sigma\diamond m)$ for each $\sigma\in \omega^{<\omega}$ and $n,m\in\omega$ with $n\leq m$. A simple inductive argument shows that $r(\sigma\diamond n)\geq n$ for all $\sigma\in \omega^{<\omega}$ and $n\in\omega$. Let $\rho\colon\omega\to\omega^{<\omega}$ be the inverse of $r$. We also fix a bijection $\phi\colon \omega\to X$.


In the procedure described below, we will construct:
\begin{enumerate}[label=(\roman*)]
\item\label{processCond1}
a (possibly finite) decreasing sequence $A_0, A_1,\ldots$ of non-empty $\bdelta 2$-subsets of $X$ such that $\phi(k)\not\in A_{k+1}$ whenever $A_{k+1}$ is defined, and
\item\label{processCond2}
for each $k$ at which $A_k$ is defined, a (possibly finite) sequence of pairs 
\[(x^k_{\rho(0)},W^k_{\rho(0)}), (x^k_{\rho(1)},W^k_{\rho(1)}),\ldots\]
with $W^k_\sigma \in\bdelta 2(X)$ and $x^k_\sigma\in W^k_\sigma\subseteq A_k$. The pairs will be constructed in the order listed above, but it will be convenient to subindex them using elements of $\omega^{<\omega}$. Furthermore, the sequence will be constructed in such a way that when we define $x^k_{\rho(t)}$ we have
\[x^k_\tau \leq x^k_{\tau'} \text{ if and only if } \tau' \preceq \tau\]
for each $\tau,\tau'\in\omega^{<\omega}$ satisfying $r(\tau),r(\tau') \leq t$ (hence $x^k_\tau$ and $x^k_{\tau'}$ are already defined).
\end{enumerate}
The procedure begins by defining $A_0 = X$, which is trivially $\bdelta 2$, and going to step $(0,0)$.
\begin{description}
\item[Step $(k,0)$] Fix any element $x^k_\varepsilon$ from $A_k$ and set $W^k_\varepsilon=A_k$. Clearly the conditions listed in \ref{processCond1} and \ref{processCond2} above are maintained. Go to step $(k,1)$.
\item[Step $(k,t)$ $(t>0)$] Let $\sigma\diamond n = \rho(t)$ and define $R =\{\tau\in\omega^{<\omega} \,|\, r(\tau)< t \mbox{ \& } \tau\not\preceq \sigma\}$.  We can assume that $x^k_{\rho(s)}$ and $W^k_{\rho(s)}$ have been defined for all $s<t$. In particular, $x^k_\sigma$ and $W^k_\sigma$ are defined, and $W^k_\sigma \in\bdelta 2(X)$ and $x^k_\sigma\in W^k_\sigma\subseteq A_k$ hold by the conditions in \ref{processCond2} above. Define
\[W^k_{\sigma\diamond n} = W^k_{\sigma} \cap  Cl(\{ x^k_\sigma \}) \cap B(x^k_\sigma,t) \setminus \bigcup_{\tau\in R} Cl(\{x^k_{\tau}\}).\]
Thus $W^k_{\sigma\diamond n}\in\bdelta 2(X)$ and $W^k_{\sigma\diamond n} \subseteq W^k_\sigma\subseteq A_k$. Furthermore, for any $\tau\in R$ we have $r(\sigma),r(\tau)< t$ and $\tau \not\preceq \sigma$, hence $x^k_\sigma \not\leq x^k_\tau$ by the condition in \ref{processCond2}. It follows that $x^k_\sigma \in W^k_{\sigma\diamond n}$, and thus $W^k_{\sigma\diamond n}$ is infinite because it is non-empty. There are two cases:
\begin{enumerate}[label=(\alph*)]
\item\label{processCase1}
There is $x\in W^k_{\sigma\diamond n}$ that is distinct from $x^k_\sigma$ and incomparable (with respect to $\leq$) with $x^k_\tau$  for each $\tau\in R$. In this case we define $x^k_{\sigma\diamond n} = x$. We must check that the conditions in \ref{processCond2} still hold. Fix any $\tau$ with $r(\tau)< t$ (the case $r(\tau)=t$ is trivial). Then $\sigma\diamond n \not\preceq \tau$ by our choice of $r$, hence if $\tau\in R$ then $\sigma\diamond n$ and $\tau$ must be incomparable with respect to $\preceq$, and $x^k_{\sigma\diamond n}$ was chosen to be incomparable with $x^k_\tau$ with respect to $\leq$. If $\tau\not\in R$ then $\tau\preceq \sigma$ and we have $x^k_{\sigma\diamond n} < x^k_\sigma \leq x^k_\tau$. Therefore, the conditions in \ref{processCond1} and \ref{processCond2} hold. Proceed to step $(k,t+1)$.
\item\label{processCase2}
Otherwise, for each $x\in W^k_{\sigma\diamond n}$ that is distinct from $x^k_\sigma$ there is $\tau\in R$ with $x^k_\tau < x$. Choose $y \in W^k_{\sigma\diamond n}$ such that $y\not=\phi(k)$, which is possible because $W^k_{\sigma\diamond n}$ is infinite. Let $U$ be an open neighborhood of $y$ small enough that $\phi(k)\not\in U \cap Cl(\{y\})$. Define $A_{k+1} = W^k_{\sigma\diamond n} \cap U \cap Cl(\{y\})$. Note that $y\in A_{k+1}\subseteq A_k$ and $\phi(k)\not\in A_{k+1}$, hence the conditions in \ref{processCond1} and \ref{processCond2} hold. Go to step $(k+1,0)$.
\end{enumerate} 
\end{description}


Assume for a contradiction that for each $k\in\omega$ there is $t_k\in\omega$ such that the above procedure enters case \ref{processCase2} of step $(k,t_k)$. Then $A_k$ is defined for each $k\in\omega$, and $\{A_k\}_{k\in\omega}$ is an infinite decreasing sequence of $\bdelta 2$-subsets of $X$. Furthermore, the intersection $\bigcap_{k\in\omega} A_k$ is empty because $\phi(k)\not\in A_{k+1}$ and $\phi\colon \omega\to X$ is a bijection. 

Define $T\subseteq \omega^{<\omega}$ so that $q \in T$ if and only if 
\begin{enumerate}
\item
$(\forall k<|q|)[q(k)<t_k]$, and
\item
$(\forall \ell < k <|q|)[x^{\ell}_{\rho(q(\ell))} < x^{k}_{\rho(q(k))}]$,
\end{enumerate}
where $|q|$ denotes the length of $q$. 

We show by induction that for each $k\in\omega$ and $t<t_k$ there is $q\in T$ with $q(k)=t$. The case $k=0$ is trivial. Assume the claim holds for $k\in\omega$ and fix $t<t_{k+1}$. Let $\sigma\diamond n = \rho(t_k)$. Since case \ref{processCase2} holds at step $(k,t_k)$ and $x^{k+1}_t \in A_{k+1} \subseteq W^k_{\sigma\diamond n}$, there is $s < t_k$ with $x^k_{\rho(s)} < x^{k+1}_t$. By the induction hypothesis there is $q\in T$ with $q(k)=s$. Then $q\diamond t \in T$, which proves the inductive step.

It follows that $T$ is an infinite finitely branching tree, hence K\"{o}nig's lemma implies that $T$ contains an infinite path $p$. Define $y_k = x^k_{\rho(p(k))}$. Then the sequence $\{y_k\}_{k\in\omega}$ satisfies $y_k\in A_k$ and $y_k < y_{k+1}$  for each $k\in\omega$. Clearly $Y=\{y_k\,|\,k\in\omega\}$ is homeomorphic to $S_D$, and by defining 
\[ Y_k = \{y_0,\ldots, y_k\} \cup A_{k+1}\]
we have $Y = \bigcap_{k\in\omega} Y_k$ because $\{A_k\}_{k\in\omega}$ has empty intersection. It follows that $Y$ is a $\bpi 2$-subset of $X$, which contradicts Theorem~\ref{thrm:completelyBaire_characterization} and our assumption that $X$ is completely Baire.

It follows that there must be $k\in\omega$ such that the above procedure reaches step $(k,0)$ and thereafter only case \ref{processCase1} holds. This means that the procedure enters step $(k,t)$ for each $t\in\omega$, hence $x^k_\sigma$ is defined for each $\sigma \in \omega^{<\omega}$. Let $S=\{x^k_\sigma \,|\, \sigma \in \omega^{<\omega} \}$. 

We next show that the mapping $f\colon S_0 \to S$ defined as $f(\sigma)=x^k_\sigma$ is a homeomorphism. From the condition in \ref{processCond2} above, we have that $x^k_\tau \leq x^k_{\tau'}$ if and only if $\tau' \preceq \tau$, which shows that $f$ is an order isomorphism between $S_0$ and $S$ (with respect to the specialization orders). Thus for any finite $F\subseteq S_0$ and $x\in S$ we have $x \in f(Cl(F))$ iff $f^{-1}(x)\in Cl(F)$ iff $(\exists \tau\in F)[ \tau \preceq f^{-1}(x) ]$ iff $(\exists \tau\in F)[x \leq f(\tau)]$, hence $f(Cl(F)) = S\cap\bigcup_{\tau\in F} Cl(\{f(\tau)\})$. It follows that the image of each basic closed subset of $S_0$ under $f$ is a closed subset of $S$, hence $f^{-1}$ is continuous. Therefore, it only remains to prove that $f$ is continuous.

Let $B_i$ be any basic open subset of $X$. First assume $f(\varepsilon)\not\in B_i$. Then for each $\sigma\in S_0$ we have $\varepsilon \preceq \sigma$ hence $f(\sigma) \leq f(\varepsilon)$ which implies $f(\sigma)\not\in B_i$. Therefore, $f^{-1}(B_i)=\emptyset$ is open.

Next assume $f(\varepsilon)\in B_i$. Note that if $f(\sigma)\in B_i$ and either $r(\sigma)\geq i$ or $n\geq i$ holds, then $r(\sigma\diamond n) \geq \max\{r(\sigma),n\}\geq i$ which implies $W^k_{\sigma\diamond n}\subseteq  B(f(\sigma) , r(\sigma\diamond n))\subseteq B_i$, hence $f(\sigma\diamond n)\in B_i$. Thus $f(\sigma)\in B_i$ implies $f(\sigma\diamond n)\in B_i$ for all but finitely many $n$. We can also conclude that the set
\[P=\{\sigma\in S_0\,|\, f(\sigma)\in B_i \mbox{ and } (\exists n)[f(\sigma\diamond n)\not\in B_i]\}\] 
is finite because $r$ is a bijection and $\sigma \in P$ implies $r(\sigma)<i$. It follows that the set \[Q=\{\sigma\diamond n \in S_0 \,|\, f(\sigma)\in B_i \mbox{ and } f(\sigma\diamond n)\not\in B_i\}\]
is finite. Since $f(\varepsilon)\in B_i$, by again using the fact that $B_i$ is an upper set and $\tau' \preceq \tau$ implies $f(\tau)\leq f(\tau')$, it is clear that $f(\sigma)\not\in B_i$ if and only if $\sigma$ has a prefix in the finite set $Q$. Therefore, $f^{-1}(B_i)$ is open in $S_0$ because $S_0\setminus f^{-1}(B_i) = Cl(Q)$. We conclude that $S$ is homeomorphic to $S_0$.

To complete the proof we must show that $S\in\bpi 2(X)$. Given $x\in X\setminus S$, the set $Z = Cl(\{ x\}) \cap S$ is closed in $S$. Since $S$ is homeomorphic to $S_0$ and $Z$ is closed in $S$, Lemma~\ref{lem:S0_closed_maximal_discrete} implies that the subspace $D$ of elements in $Z$ that are maximal with respect to the specialization order is a discrete space and $Z = Cl(D)\cap S = \{ y\in S \,|\, (\exists z\in D)\,y\leq z\}$.

We next show that $x\not\in Cl(D)$. This is trivial if $D$ is empty. Otherwise, there is some $y\in D$ and some open $U\subseteq X$ such that $U\cap D = \{y\}$ because $D$ is discrete. Then $x\in U$ because $y\in Cl(\{x\})$. Clearly, $x \not\in Cl(\{y\})$ because otherwise we would obtain $x = y \in D\subseteq S$, contradicting the assumption $x\not\in S$. Therefore, $U\setminus Cl(\{y\})$ is an open neighborhood of $x$ that does not intersect $D$.

Since $Cl(\{ x\}) \cap S\subseteq Cl(D)$, it follows that $Cl(\{ x \}) \setminus Cl(D)$ is a locally closed set containing $x$ that is disjoint from $S$. Since $X$ is countable, $X\setminus S$ is a countable union of locally closed sets, hence $S \in\bpi 2(X)$.
\end{proof}

\begin{thm}\label{thrm:countableCompletelyBaire}
If $X$ is a countably based completely Baire $T_0$-space with countably many points, then either $X$ is quasi-Polish or else $X$ contains a $\bpi 2$-subset homeomorphic to $S_0$.
\end{thm}
\begin{proof}
Assume $X$ is not quasi-Polish. We can assume that $X$ is a subspace of some quasi-Polish space $Y$. Since $X$ is completely Baire but not quasi-Polish, Theorem \ref{thrm:Pi3_completelyBaire} implies  $X\not\in \bpi 3(Y)$, hence $X\not\in \bdelta 3 (Y)$.

It follows from Theorem \ref{thrm:CharacterizationOfCountableDelta3} that there is a non-empty $A\in\bpi 2(X)$ such that every non-empty $F\in\bdelta 2(A)$ is infinite. By Theorem~\ref{thrm:completelyBaire_characterization} and the assumption that $X$ is completely Baire, we obtain that $A$ is also completely Baire. Clearly, $A$ is countable and it is infinite because by assumption $A$ is non-empty hence $A\in\bdelta 2(A)$ is infinite.

We can now apply Lemma \ref{lem:containsS0} to obtain a subspace $S\in\bpi 2(A)$ which is homeomorphic to $S_0$. Since $A\in\bpi 2(X)$, it follows that $S\in \bpi 2(X)$ as required.
\end{proof}

\begin{cor}
Every sober space with a countable topology is quasi-Polish.
\end{cor}
\begin{proof}
Let $X$ be a sober space with a countable topology. It is well known that every sober space is a $T_0$-space, hence having a countable topology implies that $X$ has only countably many points. Then $X$ is completely Baire, because otherwise Theorem~\ref{thrm:sober_not_completelyBaire_then_S2} would imply $X$ contains a subspace homeomorphic to $S_2$, which is impossible because the topology on $S_2$ is uncountable. Finally, $X$ must be quasi-Polish, because otherwise Theorem~\ref{thrm:countableCompletelyBaire} would imply $X$ contains $S_0$, which is again impossible because the topology of $S_0$ is uncountable by Proposition~\ref{prop:S0_topology_uncountable}.
\end{proof}

\section{Co-analytic spaces}

We call a subset $A$ of a quasi-Polish space $X$ \emph{analytic} if either $A$ is empty or else there exists a continuous function $f\colon \baire\to X$ such that $A=f(\baire)$. There are several equivalent characterizations of analytic subsets of quasi-Polish spaces, just as in the case for Polish spaces. In the following, $\pi_X\colon X\times Y\to X$ denotes the projection onto $X$.

\begin{lem}[see \cite{debrecht2013}]\label{lem:analytic_equiv}
The following are equivalent for any subset $A$ of a quasi-Polish space $X$:
\begin{enumerate}
\item
$A$ is analytic,
\item
$A=\pi_X(F)$ for some $\bpi 2$ subset $F\subseteq X\times \baire$,
\item
$A=\pi_X(B)$ for some quasi-Polish $Y$ and Borel subset $B\subseteq X\times Y$,
\item
$A=f(Y)$ for some quasi-Polish $Y$ and continuous $f\colon Y\to X$.\qed

\end{enumerate}
\end{lem}

Complements of analytic sets are called \emph{co-analytic}. The set of analytic subsets of a quasi-Polish space $X$ will be denoted $\analytic(X)$, and the set of co-analytic subsets will be denoted $\coanalytic(X)$. Souslin's Theorem extends to quasi-Polish spaces, which means that $\analytic(X)\cap\coanalytic(X)$ is equal to the set of Borel subsets of $X$ (see Theorem~58 of \cite{debrecht2013}). 

We now prove the main result of this paper, which is a generalization of a theorem by Hurewicz \cite{hurewicz}.

\begin{thm}[Generalized Hurewicz's Theorem]\label{thrm:main_theorem}
A co-analytic subset of a quasi-Polish space is either quasi-Polish or else it contains a $\bpi 2$-subset homeomorphic to either $S_0$, $S_D$, $S_1$, or $S_2$. 
\end{thm}
\begin{proof}
Assume $X$ is a co-analytic subset of a quasi-Polish space $Y$. If $X$ is not completely Baire, then Theorem \ref{thrm:completelyBaire_characterization} implies that some $\bpi 2$ subset of $X$ is homeomorphic to $S_2$, $S_1$, or $S_D$. Therefore, we can assume without loss of generality that $X$ is completely Baire.

Fix a basis $\{B_i\}_{i\in\omega}$ for $Y$, and let $\widehat{Y}$ be the space obtained by refining the topology on $Y$ with the sets $Y\setminus B_i$ for $i\in\omega$. The bijection $f\colon Y\to \widehat{Y}$, which acts as the identity on the points of $Y$, is clearly $\bsigma 2$-measurable (i.e., preimages of open sets are $\bsigma 2$), and the inverse $f^{-1}$ is continuous. Thus $\widehat{X} = f(X)$ is a co-analytic subset of $\widehat{Y}$. Furthermore, $\widehat{Y}$ is a zero-dimensional countably based Hausdorff space, hence metrizable, so Theorem 75 of \cite{debrecht2013} implies $\widehat{Y}$ is Polish. If $\widehat{X}\in \bpi 2(\widehat{Y})$, then $X\in \bpi 3(Y)$ follows from $f$ being $\bsigma 2$-measurable, hence Theorem \ref{thrm:Pi3_completelyBaire} implies $X$ is quasi-Polish. 

So assume $\widehat{X}\not\in \bpi 2(\widehat{Y})$. Then $\widehat{X}$ is a co-analytic subset of a Polish space which is not Polish, hence the original version of Hurewicz's theorem implies that $\widehat{X}$ contains a relatively closed subset $\widehat{C}$ which is homeomorphic to $S_2$. Thus, $C=f^{-1}(\widehat{C})$ is in $\bpi 2(X)$ because $f$ is $\bsigma 2$-measurable. Theorem~\ref{thrm:completelyBaire_characterization} and our assumption that $X$ is completely Baire implies that $C$ is a countable completely Baire space, but another application of Theorem 75 from \cite{debrecht2013} shows that $C$ cannot be quasi-Polish. By Theorem \ref{thrm:countableCompletelyBaire}, there is $S\in\bpi 2(C)$ which is homeomorphic to $S_0$, and clearly $S$ is a $\bpi 2$-subset of $X$ because $C\in\bpi 2(X)$.
\end{proof}

The co-analytic criterion can not be removed in general. According to \cite{kechris}, it is independent of ZFC whether or not every metrizable completely Baire space in the projective hierarchy is Polish. Therefore, it is consistent with ZFC to replace ``co-analytic'' in the above theorem with any other level of the projective hierarchy. However, with heavy use of the axiom of choice it is possible within ZFC to construct a metrizable completely Baire space that is not Polish (this is Exercise 21.20 in \cite{kechris}, but an explicit construction of such a space can be found in the proof of Theorem 3.7 in \cite{tillplewe}).

Clearly, a $T_D$-space cannot contain $S_0$ as a subspace, so we obtain the following.

\begin{cor}\label{cor:TD_completelybaire}
Every completely Baire co-analytic $T_D$-subspace of a quasi-Polish space is quasi-Polish.
\qed
\end{cor}

The co-analytic condition can be removed if we further assume that the space is $T_1$ and \emph{strong Choquet} (see \cite{kechris} for a definition).

\begin{cor}\label{cor:T1_completelyStrongChoquet}
Every countably based $T_1$-space which is completely Baire and strong Choquet is quasi-Polish.
\end{cor}
\begin{proof}
Mummert and Stephan \cite{mummert_stephan} proved that every countably based strong Choquet space satisfying the $T_1$-axiom can be embedded into an $\omega$-continuous domain as its set of maximal elements (see also Corollary~60 of \cite{debrecht2013}). Corollary \ref{cor:TD_completelybaire} now applies because the set of maximal elements of a quasi-Polish space is co-analytic (Theorem~59 of \cite{debrecht2013}).
\end{proof}

\section*{Acknowledgments}
The author thanks Victor Selivanov and Matthias Schr\"{o}der for helpful discussions concerning the space $S_0$. The author is also very grateful to the reviewers for their useful comments and corrections, in particular the reviewer who noticed a mistake in the proof of an earlier version of Lemma~\ref{lem:containsS0}. This work was supported by JSPS Core-to-Core Program, A. Advanced Research Networks and JSPS KAKENHI Grant Number 15K15940.

\bibliographystyle{amsplain}
\bibliography{myrefs}

\end{document}